\documentclass[a4paper,reqno,10pt]{amsart}
\usepackage{amsfonts}
\usepackage{amsmath}
\usepackage{amsthm, amsmath, amssymb, enumerate}
\usepackage{amssymb, enumitem}
\usepackage{bbm}
\usepackage{amscd}
\usepackage[all]{xy}
\usepackage[left=2cm,right=2cm,bottom=3cm,top=3cm]{geometry}
\usepackage[hidelinks]{hyperref}
\usepackage{amsmath}
\usepackage{amssymb}
\usepackage{amsthm}
\usepackage[sort&compress,numbers]{natbib}

\newtheorem{theorem}{Theorem}[section]
\newtheorem*{theorem*}{Theorem}

\newtheorem*{claim*}{Claim}

\newtheorem{proposition}[theorem]{Proposition}

\theoremstyle{definition}

\newtheorem{remark}[theorem]{Remark}
\newtheorem{lemma}[theorem]{Lemma}
\newtheorem{definition}[theorem]{Definition}
\AtBeginDocument{   \def\MR#1{}
}

\begin{document}
\def\cprime{$'$}

\title[Order-theoretic characterization of the WEP]{An intrinsic order-theoretic characterization\\
of the weak expectation property}
\author{Martino Lupini}
\address{Mathematics Department\\
California Institute of Technology\\
1200 E. California Blvd\\
MC 253-37\\
Pasadena, CA 91125}
\email{lupini@caltech.edu}
\urladdr{http://www.lupini.org/}
\thanks{The author was partially supported by the NSF Grant DMS-1600186.}
\dedicatory{}

\begin{abstract}
We prove the following characterization of the weak expectation property for operator systems in terms of Wittstock's  matricial Riesz separation property: an operator system $S$ satisfies the weak expectation property if and only if $M_{q}\left( S\right) $ satisfies the matricial Riesz separation property for every $q\in \mathbb{N}$. This can be seen as the noncommutative analog of the characterization of simplex spaces among function systems in terms of the classical Riesz separation property.
\end{abstract}

\maketitle

\section{Introduction\label{Section:intro}}

An \emph{operator system} is a closed subspace of a unital C*-algebra that
contains the unit and is invariant under the involution. Operator systems
can be seen as the noncommutative analog of compact convex sets. Indeed, one
can associate with any compact convex set the operator system $A(K)\subset
C(K)$ of continuous affine complex-valued functions on $K$.\ The operator
systems arising in this way are called \emph{function systems}, and are
precisely the operator systems that can be represented within an abelian
unital C*-algebra. The study of operator systems goes as far back as the
works of Arveson in the late 1960s \cite%
{arveson_subalgebras_1969,arveson_subalgebras_1972} and of Choi--Effros in
the 1970s \cite%
{choi_nuclear_1978,choi_lifting_1977,choi_injectivity_1977,choi_lifting_1977,choi_completely_1976}
from the 1970s. Since then, it has been pursued by many authors \cite%
{wittstock_ein_1981,wittstock_matrix_1984,wittstock_extension_1984,schmitt_characterization_1982,paulsen_completely_1982,paulsen_completely_1984,loebl_remarks_1981,hopenwasser_extreme_1981,farenick_pure_2004,webster_krein-milman_1999,effros_operator_1997,winkler_non-commutative_1999,farenick_c*-extreme_1997,farenick_c*-extreme_1993,farenick_operator_2012,farenick_operator_2014,farenicK_extremal_2000,davidson_choquet_2015,arveson_noncommutative_2008,arveson_noncommutative_2011,arveson_noncommutative_2010,kavruk_tensor_2011}%
; see also the monographs \cite%
{pisier_introduction_2003,effros_operator_2000,paulsen_completely_2002} .

The program of studying \textquotedblleft noncommutative
order\textquotedblright\ within the framework of operator systems has been
explicitly proposed by Effros in 1977 \cite{effros_aspects_1978}. Towards
this goal, in view of the key role that \emph{Choquet simplices} play in
classical convexity theory, it is important to understand what is the
correct noncommutative analog of Choquet simplices. This can be a
challenging problem, as many different equivalent characterization of
Choquet simplices exist. Several of such characterizations admit natural
noncommutative analogues, which are not obviously equivalent. It might
therefore seem arbitrary to decide that any given one of such noncommutative
analogues is the \textquotedblleft right\textquotedblright\ notion of
noncommutative Choquet simplex. However, it is arguable that, if such a
noncommutative analogues turn out to be equivalent, at least under some
generous assumptions, then the corresponding notion of \textquotedblleft
noncommutative Choquet simplex\textquotedblright\ would be sufficiently
robust and, in a sense, the correct one.

First of all, in order to have hopes of a well-behaved theory, or at least
of a theory where as many results as possible from the commutative case can
be transferred, one should restrict to the class of \emph{exact }operator
systems. Indeed, (non)exactness has no analogue in the commutative setting,
and any function system is automatically exact. The result of Junge and
Pisier \cite{junge_bilinear_1995} that the class of finite-dimensional
operator systems with respect to the completely bounded Banach--Mazur
distance is \emph{not }separable (unlike the class of finite-dimensional 
\emph{exact }operator systems) also suggests that the class of nonexact of
operator systems is too wild.

The main result of this paper is that, within the class of exact operator
systems, three natural noncommutative analogues of the notion of Choquet
simplex of very different nature---operator-theoretic, order-theoretic, and
model-theoretic---are in fact equivalent.\ The first characterization of
Choquet simplices that we consider can be formulated in terms of the \emph{%
completely positive approximation property}. A compact convex set $K$ is a
Choquet simplex if and only if the identity map of $A(K)$ is the pointwise
limit of unital completely positive maps that factor through
finite-dimensional injective operator systems (which in this case can be
chosen to be function systems). This definition applies equally well to the
noncommutative setting, yielding the important notion of \emph{nuclear }%
operator systems. It is hard to understate the importance of nuclearity for
the theory of operator algebras and operator systems. For instance, this
notion plays a crucial role in the Choi--Effros solution to the completely
positive lifting problem for C*-algebras \cite{choi_completely_1976}.
Several equivalent reformulations of nuclearity are already known; see for
instance \cite%
{kavruk_quotients_2013,kavruk_nuclearity_2014,han_approximation_2011}.

The second characterization of Choquet simplices that we consider is
order-theoretic. In classical convexity, this can be stated in terms of the 
\emph{Riesz separation property} for the space $A(K)$ as an ordered vector
space \cite{alfsen_compact_1971}. A \emph{matricial }version of the Riesz
separation property for operator systems has been considered by Wittstock in 
\cite{wittstock_ein_1981,wittstock_matrix_1984,wittstock_extension_1984}.
This property, formulated in terms of the notion of matrix sublinear
functional introduced therein, plays a crucial role in Wittstock's proof of
the decomposition and extension theorems for completely bounded maps
(obtained independently by Haagerup and by Paulsen). Along the way,
Wittstock proved that \emph{for dual operator systems} the matricial Riesz
separation property is equivalent to injectivity \cite[Satz 3.6]%
{wittstock_ein_1981}; see also \cite[Theorem 2.4]%
{schmitt_characterization_1985}. The main result of this paper is that one
can, more generally, provide a similar intrinsic order-theoretic
characterization of nuclearity. Precisely, an exact operator system $S$ is
nuclear if and only if, for every $q\in \mathbb{N}$, $M_{q}(S)$ endowed with
its canonical operator system structure satisfies the matricial Riesz
separation property.

In fact, our proof applies to arbitrary (not necessarily exact) operator
systems, in which case it provides a characterization of the weak
expectation property. Thus, an operator system $S$ satisfies the weak
expectation property if and only if $M_{q}\left( S\right) $ satisfies the
matricial Riesz separation property for every $q\in \mathbb{N}$. Other
characterizations of the weak expectation property have been obtained in 
\cite{farenick_characterisations_2013,kavruk_weak_2012} in the case of
C*-algebras.\ Particularly, in \cite[Theorem 7.4]{kavruk_weak_2012} the weak
expectation property for C*-algebras is shown to be equivalent to the
\textquotedblleft complete tight Riesz interpolation
property\textquotedblright\ introduced therein \cite[Definition 7.1]%
{kavruk_weak_2012}. The complete tight Riesz interpolation property, while
superficially similar, is very different in spirit from Wittstock's
matricial Riesz separation property. For instance, while the matricial Riesz
separation property is an \emph{intrinsic }notion, which only refers to the
given operator system $S$ itself and its matricial positive cones, the
complete tight Riesz interpolation property is defined in terms of a
concrete realization of $S$ as a space of operators. While different in
spirit, these notions turn out to be in fact equivalent, as we deduce from
our main results.

The third and last characterization of Choquet simplices that we consider is
model-theoretic. It asserts that a compact convex set $K$ is a Choquet
simplex if and only if the corresponding function system $A(K)$ is \emph{%
positively existentially closed }within the class of function systems \cite[%
Subsection 6.4]{lupini_fraisse_2015}. This means that if one can find in a
function system $V$ containing $A(K)$ a certain \emph{configuration} defined
in terms of conditions of the form $\left\Vert p(\bar{x})\right\Vert <r$,
where $p(\bar{x})$ is a degree $1$ *-polynomial with coefficients from $A(K)$
and $r\in \mathbb{R}$, then one can also find such a configuration within $%
A(K)$. We show that, similarly, an operator system satisfies the weak
expectation property if and only if it is positively existentially closed%
\emph{\ }within the class of operator systems. This is defined as above,
where one considers configurations defined in terms of norms of \emph{%
matrices }of degree $1$ *-polynomials; see Section \ref{Section:model}. In
the case of C*-algebras, this recovers a result of Goldbring and Sinclair
from \cite{goldbring_omitting_2015}. The model-theoretic notion of
positively existentially closed structure, and the related notion of
positively existential embedding, have recently found several applications
to the study of C*-algebras and C*-dynamics, as in the work of Goldbring and
Sinclair on the Kirchberg embedding problem \cite{goldbring_kirchbergs_2015}
and the works of Barlak and Szabo \cite{barlak_sequentially_2016} and
Gardella and the author \cite{gardella_equivariant_2016} providing a unified
approach to several preservation results in C*-dynamics for actions of
compact groups with finite Rokhlin dimension. This perspective has also been
used to generalize these preservation results to the case of compact \emph{%
quantum }groups \cite{barlak_spatial_2017,gardella_rokhlin_2017}.

The rest of this paper is divided into two sections. In Section \ref%
{Section:WEP} we recall some terminology concerning \textquotedblleft
noncommutative order\textquotedblright\ as introduced by Wittstock,
including matrix sublinear functionals and the matricial Riesz separation
property, and then we prove the order-theoretic characterization of the weak
expectation property mentioned above. In\ Section \ref{Section:model} we
describe in more detail the notion of positively existentially closed
operator system, and prove its equivalence with the weak expectation
property.

In the following, we adopt standard notations from the theory of operator
systems.\ Particularly, we denote operator systems as $S$ or $T$. We let $%
M_{q}(S)$ be the space of $q\times q$ matrices with entries in $S$, endowed
with its canonical operator system structure obtained from the
identification with $M_{n}(\mathbb{C})\otimes S$. If $\phi :S\rightarrow T$
is a linear map between operator systems, and $q\in \mathbb{N}$, then we let 
$\phi _{q}:M_{q}(S)\rightarrow M_{q}(T)$ be the corresponding \emph{%
amplification }defined by $\phi _{q}\left( \left[ x_{ij}\right] \right) =%
\left[ \phi (x_{ij})\right] $. The \emph{completely bounded norm }$%
\left\Vert \phi \right\Vert _{\mathrm{cb}}$ of $\phi $ is defined to be the
supremum of $\left\Vert \phi _{q}\right\Vert $ for $q\in \mathbb{N}$. We
denote by $\mathrm{Ball}\left( S\right) $ the closed unit ball of the
operator system $S$.

\section{An order-theoretic characterization of the weak expectation
property \label{Section:WEP}}

\subsection{The weak expectation property}

Let $S$ be an operator system, and $S^{\ast \ast }$ be its second dual. We
endow $S^{\ast \ast }$ with its canonical operator system structure, and
identify $S$ with its image under the canonical inclusion inside $S^{\ast
\ast }$. The notion of \emph{weak expectation property }for $S$ has been
introduced in \cite[Definition 6.4]{kavruk_quotients_2013} and \cite[Section
4]{effros_injectivity_2001}.

\begin{definition}
The operator system $S$ has the weak expectation property if for any
inclusion of operator systems $E\subset F$ and unital completely positive
map $\phi :E\rightarrow S$, there exists a unital completely positive map $%
\tilde{\phi}:F\rightarrow S^{\ast \ast }$ extending $\phi $.
\end{definition}

The following characterization of the weak expectation property is a
consequence of \cite[Theorem 6.1]{choi_injectivity_1977}.

\begin{lemma}
\label{Lemma:characterization}Suppose that $S$ is an operator system. Then $%
S $ has satisfies the weak expectation property if and only if for every $%
d\in \mathbb{N}$, inclusion of operator systems $E\subset M_{d}(\mathbb{C})$
and unital completely positive map $\phi :E\rightarrow S$, there exists a
unital completely positive map $\tilde{\phi}:M_{d}(\mathbb{C})\rightarrow
S^{\ast \ast }$ extending $\phi $.
\end{lemma}

The following characterization of the weak expectation property is an
consequence of Lemma \ref{Lemma:characterization} via a standard
approximation argument; see also \cite[Proposition 4.12]%
{goldbring_omitting_2015} and \cite[Proposition 2.3.8]%
{brown_c*-algebras_2008}. We present the details for the sake of
completeness.

\begin{proposition}
\label{Proposition:characterization}Suppose that $S$ is an operator system.
The following statements are equivalent:

\begin{enumerate}
\item $S$ satisfies the weak expectation property;

\item for any $d\in \mathbb{N}$, inclusion of operator systems $E\subset
M_{d}(\mathbb{C})$, unital completely positive map $\phi :E\rightarrow S$,
and $\varepsilon >0$, there exists a linear map $\psi :M_{d}(\mathbb{C}%
)\rightarrow S$ such that $\left\Vert \psi _{d}\right\Vert \leq
1+\varepsilon $, $\left\Vert \psi (1)-1\right\Vert \leq \varepsilon $, and $%
\left\Vert \psi |_{E}-\phi \right\Vert \leq \varepsilon $;

\item for any $d\in \mathbb{N}$, inclusion of operator systems $E\subset
M_{d}(\mathbb{C})$, unital completely positive map $\phi :E\rightarrow S$,
and $\varepsilon >0$, there exists a unital completely positive map $\psi
:M_{d}(\mathbb{C})\rightarrow S$ such that $\left\Vert \psi |_{E}-\phi
\right\Vert <\varepsilon $.
\end{enumerate}
\end{proposition}

\begin{proof}
The fact that (3) implies (1) is easily seen by taking a $\sigma \left(
S^{\ast \ast },S^{\ast }\right) $-limit. The implication (2)$\Rightarrow $%
(3) is a standard approximation argument; see \cite[Lemma 4.3]%
{goldbring_model-theoretic_2015}. We prove the remaining implication (1)$%
\Rightarrow $(2).

Fix $d\in \mathbb{N}$, an operator system $E\subset M_{d}(\mathbb{C})$, a
unital completely positive map $\phi :E\rightarrow S$, and $\varepsilon >0$.
By assumption there exists a unital completely positive map $\tilde{\phi}%
:M_{d}(\mathbb{C})\rightarrow S^{\ast \ast }$ such that $\tilde{\phi}%
|_{E}=\phi $. Consider the correspondence between completely positive maps $%
\psi :M_{d}(\mathbb{C})\rightarrow S^{\ast \ast }$ and positive elements $%
a_{\psi }=\left[ \psi \left( e_{ij}\right) \right] $ of $M_{d}\left( S^{\ast
\ast }\right) $. Since $M_{d}(S)$ is $\sigma \left( S^{\ast \ast },S^{\ast
}\right) $-weakly dense in $M_{d}\left( S^{\ast \ast }\right) $, there
exists a net $\left( \psi _{\lambda }\right) $ of completely positive maps $%
\psi _{\lambda }:M_{d}(\mathbb{C})\rightarrow S$ such that $\psi _{\lambda
}(x)\rightarrow \phi (x)$ in the $\sigma \left( S^{\ast \ast },S^{\ast
}\right) $-topology for every $x\in M_{d}(\mathbb{C})$. Fix $x_{1},\ldots
,x_{n}\in E$ and $\delta >0$. Consider the net%
\begin{equation*}
\eta _{\lambda }=\left( \psi _{\lambda }\left( x_{1}\right) -\phi \left(
x_{1}\right) ,\ldots ,\psi _{\lambda }\left( x_{n}\right) -\phi \left(
x_{n}\right) ,\psi _{\lambda }(1)-1\right)
\end{equation*}%
in $S\oplus \cdots \oplus S$. We have that $\left( \eta _{\lambda }\right) $
converges to $0$ weakly. It follows from the Hahn-Banach theorem that the
norm closure of the convex hull of $\left( \eta _{\lambda }\right) $
coincides with the weak closure of the convex hull of $\left( \eta _{\lambda
}\right) $. Therefore $0$ belongs to the norm closure of the convex hull of $%
\left( \eta _{\lambda }\right) $. Therefore there exists a convex
combinations $\bar{\psi}$ of the $\psi _{\lambda }$'s such that $\left\Vert 
\bar{\psi}\left( x_{i}\right) -\phi \left( x_{i}\right) \right\Vert <\delta $
for $i=1,2,\ldots ,n$ and $\left\Vert \bar{\psi}(1)-1\right\Vert <\delta $.
Since $\bar{\psi}$ is a convex combination of the $\psi _{\lambda }$'s, we
have that $\bar{\psi}$ is completely positive, and in particular $\left\Vert 
\bar{\psi}\right\Vert _{cb}=\left\Vert \bar{\psi}(1)\right\Vert \leq
1+\delta $. By \cite[Lemma 8.1]{lupini_fraisse_2015} there exists a unital
completely positive map $\psi :M_{d}(\mathbb{C})\rightarrow S$ such that $%
\left\Vert \psi -\bar{\psi}\right\Vert \leq 2d^{2}\delta $. By choosing $%
\delta >0$ small enough, one can ensure that $\left\Vert \psi |_{E}-\phi
\right\Vert \leq \varepsilon $. This concludes the proof.
\end{proof}

Recall that an operator system $S$ is \emph{exact }if for any
finite-dimensional subspace $E$ of $S$ and $\varepsilon >0$ there exist $%
d\in \mathbb{N}$, a subspace $F$ of $M_{d}(\mathbb{C})$, and unital
completely positive maps $\phi :E\rightarrow F$ and $\psi :F\rightarrow E$
such that $\left\Vert \psi \circ \phi -\mathrm{id}_{E}\right\Vert
<\varepsilon $ and $\left\Vert \psi \circ \phi -\mathrm{id}_{F}\right\Vert
<\varepsilon $. An operator system $S$ is \emph{nuclear }if it satisfies the 
\emph{completely positive approximation property}, namely there exists nets $%
\left( \rho _{i}\right) ,\left( \gamma _{i}\right) $ of unital completely
positive maps $\gamma _{i}:S\rightarrow M_{d_{i}}(\mathbb{C})$ and $\rho
_{i}:M_{d_{i}}(\mathbb{C})\rightarrow S$ such that $\left\Vert \left( \rho
_{i}\circ \gamma _{i}\right) (x)-x\right\Vert \rightarrow 0$ for every $x\in
S$. It follows from Proposition \ref{Proposition:characterization} that an
operator system is nuclear if and only if it is exact and it satisfies the
weak expectation property.

\subsection{Matrix sublinear functionals and noncommutative order}

We recall in these sections some notions about matrix sublinear functionals
and noncommutative order introduced by Wittstock in \cite%
{wittstock_ein_1981, wittstock_matrix_1984, wittstock_extension_1984}; see
also \cite{effros_aspects_1978, schmitt_characterization_1985}.

Suppose that $V$ is an ordered real vector spaces. For subsets $A,B$ of $V$,
we let $A\preccurlyeq B$ if and only if for every $b\in B$ there exists $%
a\in A$ such that $a\leq b$; see \cite[Definition 2.1.1]{wittstock_ein_1981}%
. We identify an element $a$ of $V$ with the corresponding singleton.
Consistently, we write $a\preccurlyeq B$ if $a\leq b$ for every $b\in B$. A $%
\ast $-vector space is a complex vector space $V$ endowed with a
conjugate-linear map $V\rightarrow V$, $x\mapsto x^{\ast }$ satisfying $%
\left( x^{\ast }\right) ^{\ast }=x$. The real subspace of elements
satisfying $x^{\ast }=x$ will be denoted by $V_{\mathrm{sa}}$. The space $%
M_{n}\left( V\right) $ of $n\times n$ matrices with entries in $V$ has a
canonical $\ast $-vector space structure obtained by setting $\left[ x_{ij}%
\right] ^{\ast }=\left[ x_{ji}^{\ast }\right] $. A $\ast $-vector space is
ordered if it is endowed with a distinguished proper convex cone $%
V^{+}\subset V_{\mathrm{sa}}$, and matrix ordered if it is endowed with
distinguished proper convex cones $M_{n}\left( V\right) ^{+}\subset
M_{n}\left( V\right) _{\mathrm{sa}}$ such that $\gamma ^{\ast }M_{n}\left(
V\right) ^{+}\gamma \subset M_{m}\left( V\right) ^{+}$ for $n,m\in \mathbb{N}
$ and $\gamma \in M_{n,m}(\mathbb{C})$. Any operator system is endowed with
a canonical matrix ordered $\ast $-vector space structure.

The notion of matrix sublinear functional has been introduced by Wittstock
in \cite[Definition 2.1.1]{wittstock_ein_1981}.

\begin{definition}
\label{Definition:matrix-sublinear}Let $V$ be a $\ast $-vector space and $W$
be an operator system. Suppose that $\theta $ is a sequence $\left( \theta
_{n}\right) _{n\in \mathbb{N}}$, where $\theta _{n}$ is a function that
assigns to any element $v$ of $M_{n}\left( V\right) _{\mathrm{sa}}$ a subset 
$\theta _{n}\left( v\right) $ of $M_{n}\left( W\right) _{\mathrm{sa}}$. Then 
$\theta :V\rightarrow W$ is a \emph{matrix sublinear functional }if, for
every $u,v\in M_{n}\left( V\right) _{\mathrm{sa}}$, $\gamma \in M_{n,m}(%
\mathbb{C})$, and $n,m\in \mathbb{N}$, one has the following:

\begin{enumerate}
\item $\theta _{n}\left( v\right) $ is nonempty,

\item $\theta _{n}\left( u+v\right) \preccurlyeq \theta _{n}\left( u\right)
+\theta _{n}\left( v\right) $,

\item $0\preccurlyeq \theta _{n}\left( 0\right) $,

\item $\theta _{m}\left( \gamma ^{\ast }v\gamma \right) \preccurlyeq \gamma
^{\ast }\theta _{n}\left( v\right) \gamma $.
\end{enumerate}

We say that a matrix sublinear functional $\theta :V\rightarrow W$ is \emph{%
completely positive} if $0\preccurlyeq \theta _{n}\left( v\right) $ for
every $n\in \mathbb{N}$ and $v\in M_{n}\left( V\right) _{\mathrm{sa}}$.
\end{definition}

Further properties of matrix sublinear functionals are listed in \cite[Lemma
2.1.3]{wittstock_ein_1981}. Observe that a (completely positive) selfadjoint
linear map $\phi :V\rightarrow W$ can be regarded as a (completely positive)
matrix sublinear functional in the obvious say.

Suppose that $S$ is an operator system, and $\alpha \in M_{n}(\mathbb{C})_{%
\mathrm{sa}}$. An element $x$ of $M_{n}(S)_{\mathrm{sa}}$ is $\alpha $-\emph{%
positive} if, whenever $\gamma _{1},\ldots ,\gamma _{\ell }\in M_{n,m}(%
\mathbb{C})$ for $\ell ,m\in \mathbb{N}$ are such that $\gamma _{1}^{\ast
}\alpha \gamma _{1}+\cdots +\gamma _{\ell }^{\ast }\alpha \gamma _{\ell }=0$%
, one has that $\gamma _{1}^{\ast }x\gamma _{1}+\cdots +\gamma _{\ell
}^{\ast }x\gamma _{\ell }\geq 0$ \cite[Definition 3.1(a)]{wittstock_ein_1981}%
; see also \cite[Definition 2.2]{schmitt_characterization_1985}. We say that 
$x\in M_{n}(S)_{\mathrm{sa}}$ is \emph{strictly} $\alpha $-\emph{positive}
if, whenever $\gamma _{1},\ldots ,\gamma _{\ell }\in M_{n,m}(\mathbb{C})$
for $\ell ,m\in \mathbb{N}$ are such that $\gamma _{1}^{\ast }\alpha \gamma
_{1}+\cdots +\gamma _{\ell }^{\ast }\alpha \gamma _{\ell }\geq 0$ one has
that $\gamma _{1}^{\ast }x\gamma _{1}+\cdots +\gamma _{\ell }^{\ast }x\gamma
_{\ell }\geq 0$. An element $v$ of $S_{\mathrm{sa}}$ is a lower $\alpha $%
-bound for $x$ if $v\otimes \alpha \leq x$ \cite[Definition 3.1(b)]%
{wittstock_ein_1981}; see also \cite[Definition 2.2]%
{schmitt_characterization_1985}.

\begin{remark}
\label{Remark:sum}Suppose that $\alpha _{i}\in M_{n_{i}}(\mathbb{C})_{%
\mathrm{sa}}$ and $x_{i}\in M_{n_{i}}(S)_{\mathrm{sa}}$ is (strictly) $%
\alpha _{i}$-positive for $i=1,2,\ldots ,n$. Set $n:=n_{1}+\cdots +n_{\ell }$%
, $\alpha :=\alpha _{1}\oplus \cdots \oplus \alpha _{\ell }\in M_{n}(\mathbb{%
C})_{\mathrm{sa}}$, and $x:=x_{1}\oplus \cdots \oplus x_{\ell }\in M_{n}(S)_{%
\mathrm{sa}}$. Then $x$ is (strictly) $\alpha $-positive.
\end{remark}

We let $\sigma _{m,n}$ be the matrix $1_{m}\oplus \left( -1_{n}\right) \in
M_{m+n}(\mathbb{C})$, where $1_{d}\in M_{d}(\mathbb{C})$ is the $d\times d$
identity matrix for $d\in \mathbb{N}$. The matricial Riesz separation
property has been introduced by Wittstock in \cite[Definition 3.1]%
{wittstock_ein_1981}; see also \cite[Definition 2.2]%
{schmitt_characterization_1985}.

\begin{definition}[Wittstock]
\label{Definition:matricial-Riesz}An operator system $S$ satisfies the \emph{%
matricial Riesz separation property} if for every $n\in \mathbb{N}$, every $%
\sigma _{n,n}$-positive $x\in M_{2n}(S)_{\mathrm{sa}}$ has a lower $\sigma
_{n,n}$-bound $v\in S_{\mathrm{sa}}$.
\end{definition}

The following notion has also been considered by Wittstock in \cite[Theorem
2.3]{wittstock_matrix_1984} under the name of matricial Riesz separation
property. To distinguish it from the property given by Definition \ref%
{Definition:matricial-Riesz}, we call it the positive matricial Riesz
separation property.

\begin{definition}
\label{Definition:positive-matricial-Riesz}An operator system $S$ satisfies
the \emph{positive} matricial Riesz separation property if for every $n\in 
\mathbb{N}$, every strictly $\sigma _{n,n}$-positive $x\in M_{2n}(S)_{%
\mathrm{sa}}$ has a lower $\sigma _{n,n}$-bound $v\in S^{+}$.
\end{definition}

It is remarked at the end of Section 2.1 in \cite[Theorem 2.3]%
{wittstock_matrix_1984} that \textquotedblleft it does not seem to be
obvious that both properties [considered in Definition \ref%
{Definition:matricial-Riesz} and Definition \ref%
{Definition:positive-matricial-Riesz}] are equivalent\textquotedblright .

The matricial Riesz separation property can be seen as the noncommutative
analog of the Riesz separation property for an ordered vector space. Recall
that a real ordered vector space $V$ satisfies the Riesz separation property
if for any $n\in \mathbb{N}$, $x_{1},\ldots ,x_{n},y_{1},\ldots ,y_{n}\in V$
such that $x_{i}\leq y_{j}$ for $i,j\in \left\{ 1,2,\ldots ,n\right\} $
there exists $z\in V$ such that $x_{i}\leq z\leq y_{j}$ for $i,j\in \left\{
1,2,\ldots ,n\right\} $. We say that a $\ast $-vector space $V$ satisfies
the Riesz separation property if $V_{\mathrm{sa}}$ does.

Suppose now that $S$ is a function system, i.e.\ an operator system that can
be represented inside a commutative C*-algebra. In this case, by the Kadison
representation theorem \cite[Theorem II.1.8]{alfsen_compact_1971}, $S$ is
unitally completely order isomorphic to the space $A(K)$ of complex-valued
continuous affine functions on $K$, where $K$ is the state space of $S$
endowed with the w*-topology. It is well known that in this case $A(K)$
satisfies the Riesz separation property if and only if $K$ is a Choquet
simplex \cite[Corollary II.3.11]{alfsen_compact_1971}.\ It is proved in \cite%
[Proposition 2.1]{wittstock_matrix_1984} that $A(K)$ satisfies the Riesz
separation property if and only if it satisfies the positive matricial Riesz
separation property from Definition \ref{Definition:positive-matricial-Riesz}%
. In fact, one can directly prove the following.

\begin{proposition}
\label{Proposition:commutative}Suppose that $S=A(K)$ is a function system
with state space $K$. The following assertions are equivalent.

\begin{enumerate}
\item $S$ satisfies the Riesz separation property;

\item $S$ satisfies the matricial Riesz separation property;

\item $S$ satisfies the positive matricial Ries separation property;

\item $M_{q}(S)$ satisfies the matricial Riesz separation property for every 
$q\in \mathbb{N}$;

\item $M_{q}(S)$ satisfies the positive matricial Riesz separation property
for every $q\in \mathbb{N}$.
\end{enumerate}
\end{proposition}

\begin{proof}
(1)$\Rightarrow $(4): Suppose that $S$ satisfies the Riesz separation
property, in which case $K$ is a Choquet simplex by \cite[Corollary II.3.11]%
{alfsen_compact_1971}. Fix $q,n\in \mathbb{N}$, and a $\sigma _{n,n}$%
-positive element $a\in M_{2n}\left( M_{q}(S)\right) _{\mathrm{sa}}$.
Observe that $M_{2n}\left( M_{q}(S)\right) $ can be identified with the
space of $M_{2n}\left( M_{q}(\mathbb{C})\right) $-valued continuous affine
functions on $K$. Consistently, we regard $a$ as a function from $K$ to $%
M_{2n}\left( M_{q}(\mathbb{C})\right) _{\mathrm{sa}}$, and denote by $%
a\left( p\right) $ the value of $a$ at $p\in K$. Fix $p\in K$. Since $a$ is $%
\sigma _{n,n}$-positive, $a\left( p\right) \in M_{2n}\left( M_{q}(\mathbb{C}%
)\right) _{\mathrm{sa}}$ is $\sigma _{n,n}$-positive. Therefore by \cite[%
Theorem 2.4]{wittstock_matrix_1984} there exists $\alpha \in M_{q}(\mathbb{C}%
)_{\mathrm{sa}}$ such that $\alpha \otimes \sigma _{n,n}\leq a\left(
p\right) $. This shows that the set%
\begin{equation*}
\Phi \left( p\right) :=\left\{ \alpha \in M_{q}(\mathbb{C})_{\mathrm{sa}%
}:\alpha \otimes \sigma _{n,n}\leq a\left( p\right) \right\}
\end{equation*}%
is a nonempty closed convex subset of $M_{q}(\mathbb{C})_{\mathrm{sa}}$.
Since $a$ is affine, the assignment $p\mapsto \Phi \left( p\right) $
satisfies%
\begin{equation*}
\Phi \left( tp+\left( 1-t\right) q\right) \supset t\Phi \left( p\right)
+\left( 1-t\right) \Phi \left( q\right) \text{.}
\end{equation*}%
Furthermore, since $a$ is continuous, for any open subset $U$ of $M_{q}(%
\mathbb{C})_{\mathrm{sa}}$ the set%
\begin{equation*}
\left\{ p\in K:\Phi \left( p\right) \cap U\neq \varnothing \right\}
\end{equation*}%
is open. Therefore by Lazar's selection theorem \cite[Theorem 3.1]%
{lazar_spaces_1968} there exists $b\in A\left( K,M_{q}(\mathbb{C})_{\mathrm{%
sa}}\right) =M_{q}(S)_{\mathrm{sa}}$ such that $b\left( p\right) \in \Phi
\left( p\right) $ for every $p\in K$ or, equivalently, $b\otimes \sigma
_{n,n}\leq a$. This concludes the proof.

The proof of the implication (1)$\Rightarrow $(5) is entirely analogous to
the proof of the implication (1)$\Rightarrow $(4), where one replaces \cite[%
Theorem 2.4]{wittstock_matrix_1984} with \cite[Lemma 2.2]%
{wittstock_matrix_1984}. The implication (3)$\Rightarrow $(1) is observed in 
\cite[Section 2.2]{wittstock_matrix_1984}, and the implication (2)$%
\Rightarrow $(1) is analogous to the implication (3)$\Rightarrow $(1). The
rest of the implications are trivial.
\end{proof}

It is well known that, in the commutative case, the Riesz separation
property is equivalent to the \emph{approximate }Riesz separation property;
see \cite[Corollary II.3.11]{alfsen_compact_1971}. The natural
noncommutative analog of this fact holds as well.

\begin{lemma}
\label{Lemma:approximate}Let $S$ be an operator system. Suppose that for
every $n\in \mathbb{N}$, every $\sigma _{n,n}$-positive $x\in M_{2n}(S)_{%
\mathrm{sa}}$, and $\varepsilon >0$, the element $x+\varepsilon 1_{n}\in
M_{2n}(S)_{\mathrm{sa}}$ has a lower $\sigma _{n,n}$-bound $v\in M_{q}(S)_{%
\mathrm{sa}}$. Then $S$ satisfies the matricial Riesz separation property.
\end{lemma}

\begin{proof}
Using Remark \ref{Remark:sum}, it is easy to see that $S$ satisfies the
following property: for every $q,n_{1},\ldots ,n_{\ell }\in \mathbb{N}$, $%
\sigma _{n_{i},n_{i}}$-positive $x_{i}\in M_{2n_{i}}(S)_{\mathrm{sa}}$, and $%
\varepsilon >0$, there exists $v\in S_{\mathrm{sa}}$ such that $v\otimes
\sigma _{n_{i},n_{i}}\leq x_{i}+\varepsilon 1_{n_{i}}$ for $i=1,2,\ldots
,\ell $. Indeed, one can consider $n:=n_{1}+\cdots +n_{\ell }$, $\alpha
:=\sigma _{n_{1},n_{1}}\oplus \cdots \oplus \sigma _{n_{\ell },n_{\ell }}\in
M_{2n}(\mathbb{C})_{\mathrm{sa}}$, and $x:=x_{1}\oplus \cdots \oplus x_{\ell
}\in M_{2n}(S)_{\mathrm{sa}}$. Then $x$ is $\alpha $-positive, and $\alpha $
is unitarily conjugate to $\sigma _{2n,2n}$. Therefore by assumption there
exists $v\in S_{\mathrm{sa}}$ such that $v\otimes \alpha \leq x+\varepsilon
1_{2n}$. Hence $v\otimes \sigma _{n_{i},n_{i}}\leq x_{i}+\varepsilon
1_{n_{i}}$ for $i=1,2,\ldots ,\ell $.

Fix now $n\in \mathbb{N}$ and a $\sigma _{n,n}$-positive element $x\in
M_{2n}(S)$. We want to show that there exists $v\in S_{\mathrm{sa}}$ such
that $v\otimes \sigma _{n,n}\leq x$. By assumption, there exists a $v_{0}\in
S_{\mathrm{sa}}$ such that $v_{0}\otimes \sigma _{n,n}\leq x+1_{n}$. Observe
now that $v_{0}\otimes \sigma _{1,1}$ is $\sigma _{1,1}$-positive. Therefore
by assumption there exists $v_{1}\in S_{\mathrm{sa}}$ such that $%
v_{1}\otimes \sigma _{n,n}\leq x+2^{-1}1_{n}$ and $v_{1}\otimes \sigma
_{1,1}\leq v_{0}\otimes \sigma _{1,1}+2^{-1}1_{2}$. Proceeding in this way,
one can define at the $k$-th step an element $v_{k}\in S_{\mathrm{sa}}$ such
that $v_{k}\otimes \sigma _{n,n}\leq x+2^{-k}1_{n}$ and $v_{k}\otimes \sigma
_{1,1}\leq v_{k-1}\otimes \sigma _{1,1}+2^{-k}1_{2}$. Thus $\left(
v_{k}\right) $ is a Cauchy sequence in $M_{q}(S)_{\mathrm{sa}}$ converging
to an element $v$ of $M_{q}(S)_{\mathrm{sa}}$ such that $v\otimes \sigma
_{n,n}\leq x$.
\end{proof}

The same argument applies to the case of the positive matricial Riesz
separation property, and gives the following lemma.

\begin{lemma}
\label{Lemma:positive-approximate}Let $S$ be an operator system. Suppose
that for every $n\in \mathbb{N}$, every strictly $\sigma _{n,n}$-positive $%
x\in M_{2n}(S)_{\mathrm{sa}}$, and $\varepsilon >0$, the element $%
x+\varepsilon 1_{n}\in M_{2n}(S)_{\mathrm{sa}}$ has a lower $\sigma _{n,n}$%
-bound $v\in M_{q}(S)^{+}$. Then $S$ satisfies the positive matricial Riesz
separation property.
\end{lemma}

\subsection{The proof of the main theorem}

The goal of this subsection is to prove the following characterization of
the weak expectation property for operator systems in terms of Wittstock's
matricial Riesz separation property.

\begin{theorem}
\label{Theorem:nc-Riesz}Suppose that $S$ is an operator system. The
following assertions are equivalent:

\begin{enumerate}
\item $S$ satisfies the weak expectation property;

\item $M_{q}(S)$ satisfies the matricial Riesz separation property for every 
$q\in \mathbb{N}$;

\item $M_{q}(S)$ satisfies the positive matricial Riesz separation property
for every $q\in \mathbb{N}$;

\item For every $q\in \mathbb{N}$ and matrix sublinear functional $\theta :%
\mathbb{C}\rightarrow M_{q}(S)$ there exists $v\in M_{q}\left( S^{\ast \ast
}\right) _{\mathrm{sa}}$ such that $v\otimes \alpha \preccurlyeq \theta
_{n}(\alpha )$ for every $n\in \mathbb{N}$ and $\alpha \in M_{n}(\mathbb{C}%
)_{\mathrm{sa}}$;

\item For every $q\in \mathbb{N}$ and completely positive matrix sublinear
functional $\theta :\mathbb{C}\rightarrow M_{q}(S)$ there exists $v\in
M_{q}\left( S^{\ast \ast }\right) ^{+}$ such that $v\otimes \alpha
\preccurlyeq \theta _{n}(\alpha )$ for every $n\in \mathbb{N}$ and $\alpha
\in M_{n}(\mathbb{C})_{\mathrm{sa}}$;

\item For every $q\in \mathbb{N}$, $\ast $-vector space $V$, selfadjoint
subspace $W\subset V$, matrix sublinear functional $\theta :V\rightarrow
M_{q}(S)$, and selfadjoint linear map $\phi _{0}:W\rightarrow M_{q}(S)$ such
that $\phi _{0}\preccurlyeq \theta |_{W}$, there exists a selfadjoint linear
map $\phi :V\rightarrow M_{q}\left( S^{\ast \ast }\right) $ such that $\phi
\preccurlyeq \theta $ and $\phi |_{W}=\phi _{0}$.
\end{enumerate}
\end{theorem}

\begin{proof}
(1)$\Rightarrow $(2): Suppose that $S$ satisfies the weak expectation
property. Fix $q\in \mathbb{N}$. We verify that $M_{q}(S)$ satisfies the
conditions of Lemma \ref{Lemma:approximate}. Fix $n\in \mathbb{N}$, a $%
\sigma _{n,n}$-positive $x\in M_{2n}\left( M_{q}(S)\right) _{\mathrm{sa}}$,
and $\varepsilon >0$. Fix a unital complete order embedding $\eta
:S\rightarrow B\left( H\right) $. Since $S$ satisfies the weak expectation
property, there exists a unital completely positive map $\phi :B\left(
H\right) \rightarrow S^{\ast \ast }$ such that $\phi \circ \eta $ is the
inclusion map of $S$ in $S^{\ast \ast }$. By \cite[Theorem 2.4]%
{wittstock_matrix_1984}, $M_{q}\left( B\left( H\right) \right) \cong B\left(
H\oplus \cdots \oplus H\right) $ satisfies the Riesz separation property.
Therefore there exists $w\in M_{q}\left( B\left( H\right) \right) _{\mathrm{%
sa}}$ such $w\otimes \sigma _{n,n}\leq \eta _{q}(x)$. Thus $v:=\left( \phi
\circ \eta \right) _{q}\left( w\right) $ is an element of $M_{q}\left(
S^{\ast \ast }\right) _{\mathrm{sa}}$ such that $v\otimes \sigma _{n,n}\leq
x $. By $\sigma \left( S^{\ast \ast },S^{\ast }\right) $-density of $S$ in $%
S^{\ast \ast }$ together with a convexity argument as in the proof of
Proposition \ref{Proposition:characterization}, one can now find $v\in
M_{q}(S)_{\mathrm{sa}}$ such that $v\otimes \sigma _{n,n}\leq x+\varepsilon
1_{n}$.

(1)$\Rightarrow $(3): This is the same as (1)$\Rightarrow $(2), using \cite[%
Theorem 2.3]{wittstock_matrix_1984} instead of \cite[Theorem 2.4]%
{wittstock_matrix_1984} and Lemma \ref{Lemma:positive-approximate} instead
of Lemma \ref{Lemma:approximate}.

(2)$\Rightarrow $(4): The proof of this implication is analogous to the
proof of \cite[Lemma 3.4]{wittstock_ein_1981}. Suppose that $\theta :\mathbb{%
C}\rightarrow M_{q}(S)$ is a matrix sublinear functional. For $n\in \mathbb{N%
}$ and $\alpha \in M_{n}(\mathbb{C})_{\mathrm{sa}}$, set%
\begin{equation*}
L(\alpha ):=\left\{ r\in M_{q}\left( S^{\ast \ast }\right) _{\mathrm{sa}%
}:r\otimes \alpha \preccurlyeq \theta _{n}(\alpha )\right\} \text{.}
\end{equation*}%
If $m\in \mathbb{N}$ and $\gamma _{1},\ldots ,\gamma _{\ell }\in M_{n,m}(%
\mathbb{C})$ are such that $\gamma _{1}^{\ast }\alpha \gamma _{1}+\cdots
+\gamma _{\ell }^{\ast }\alpha \gamma _{\ell }=0$, then 
\begin{equation*}
0\preccurlyeq \theta _{m}\left( 0\right) =\theta _{m}\left( \gamma
_{1}^{\ast }\alpha \gamma _{1}+\cdots +\gamma _{\ell }^{\ast }\alpha \gamma
_{\ell }\right) \preccurlyeq \gamma _{1}^{\ast }\theta _{n}(\alpha )\gamma
_{1}+\cdots +\gamma _{\ell }^{\ast }\theta _{n}\left( \alpha \right) \gamma
_{\ell }\text{.}
\end{equation*}%
This shows that $\theta _{n}(\alpha )\subset M_{n}\left( M_{q}(S)\right) _{%
\mathrm{sa}}$ is $\alpha $-positive. Set $L_{m,n}:=L\left( \sigma
_{m,n}\right) $ for $m,n\in \mathbb{N}$. Observe that $L_{n,n}\subset
L_{k,\ell }$ for $0\leq k,\ell \leq n$, and $L_{n,n}$ is nonempty by
assumption. If $\alpha \in M_{n}(\mathbb{C})_{\mathrm{sa}}$, then $L(\alpha
)\subset L_{k,\ell }$ for some $k,\ell \in \mathbb{N}$. Indeed, if $\alpha $
is invertible, then $\alpha =\beta ^{\ast }\sigma _{k,\ell }\beta $ for some 
$\beta \in M_{n}(\mathbb{C})$ invertible and $k,\ell \in \mathbb{N}$. Thus
if $r\in L_{k,\ell }$ then by \cite[Lemma 2.1.3(f)]{wittstock_ein_1981}%
\begin{equation*}
r\otimes \alpha =\beta ^{\ast }\left( r\otimes \sigma _{k,\ell }\right)
\beta \preccurlyeq \beta ^{\ast }\theta _{n}\left( \sigma _{k,\ell }\right)
\beta \preccurlyeq \theta _{n}\left( \beta ^{\ast }\sigma _{k,\ell }\beta
\right) =\theta _{n}(\alpha )\text{.}
\end{equation*}%
When $\alpha $ is not necessarily invertible, one can replace $\alpha $ with 
$\alpha +\delta 1_{n}$ for $\delta >0$ and then take the limit $\delta
\rightarrow 0$. Therefore we have%
\begin{equation*}
\bigcap_{n\in \mathbb{N}}L_{n,n}=\bigcap_{n\in \mathbb{N}}\bigcap_{\alpha
\in M_{n}(\mathbb{C})}L(\alpha )\text{.}
\end{equation*}%
Such an intersection is nonempty by $\sigma \left( S^{\ast \ast },S^{\ast
}\right) $-compactness of the order intervals in $M_{q}\left( S^{\ast \ast
}\right) ^{+}$. Any element in such an intersection satisfies the desired
conclusions.

(3)$\Rightarrow $(5): This is the same as the implication (2)$\Rightarrow $%
(4).

(4)$\Rightarrow $(1): Suppose that $q\in \mathbb{N}$, $E\subset M_{q}(%
\mathbb{C})$, and $\phi :E\rightarrow S$ is a unital completely positive
map. Let $e_{ij}$ be the matrix units of $M_{q}(\mathbb{C})$, and $e=\left[
e_{ij}\right] \in M_{q}\left( M_{q}(\mathbb{C})\right) ^{+}$ be the Choi
matrix. Define a matrix sublinear functional $\rho :M_{q}(\mathbb{C}%
)\rightarrow S$ by setting, for $n\in \mathbb{N}$ and $\alpha \in
M_{n}\left( M_{q}(\mathbb{C})\right) _{\mathrm{sa}}$,%
\begin{equation*}
\rho _{n}(\alpha )=\left\{ \phi _{n}\left( \beta \right) :\beta \in
M_{n}(E),\beta \geq \alpha \right\} \subset M_{n}(S)_{\mathrm{sa}}\text{.}
\end{equation*}%
It is not difficult to verify that $\rho $ is indeed a matrix sublinear
functional. Define now a matrix sublinear functional $\theta :\mathbb{C}%
\rightarrow M_{q}(S)$ by%
\begin{equation*}
\theta _{n}(\alpha )=\rho _{nq}\left( e\otimes \alpha \right) \subset
M_{n}\left( M_{q}(S)\right) _{\mathrm{sa}}\text{.}
\end{equation*}%
By assumption, there exists $v\in M_{q}\left( S^{\ast \ast }\right) _{%
\mathrm{sa}}$ such that $v\otimes \alpha \preccurlyeq \theta _{n}(\alpha )$
for every $n\in \mathbb{N}$ and $\alpha \in M_{n}(\mathbb{C})$. In
particular we have that $-v\preccurlyeq \theta _{n}\left( -1\right) =\rho
_{q}\left( -e\right) \preccurlyeq \phi _{q}\left( -e\right) \leq 0$ and
hence $v\in M_{q}\left( S^{\ast \ast }\right) ^{+}$. Consider now the
completely positive map $\tilde{\phi}:M_{q}(\mathbb{C})\rightarrow S^{\ast
\ast }$ such that $\tilde{\phi}_{q}\left( e\right) =v$. For $\alpha \in E$
we have that%
\begin{equation*}
\tilde{\phi}_{qn}\left( e\otimes \alpha \right) =v\otimes \alpha
\preccurlyeq \theta _{n}(\alpha )=\rho _{nq}\left( e\otimes \alpha \right)
\preccurlyeq \phi _{qn}\left( e\otimes \alpha \right) \text{.}
\end{equation*}%
For $i=1,2,\ldots ,q$, let $\xi _{i}\in \mathbb{C}^{q}$ be the column vector
whose only nonzero entry is a $1$ in the $i$-th row, and then set%
\begin{equation*}
\xi =%
\begin{bmatrix}
\xi _{1} \\ 
\xi _{2} \\ 
\vdots  \\ 
\xi _{q}%
\end{bmatrix}%
\in \mathbb{C}^{q^{2}}\text{.}
\end{equation*}%
Then we have that%
\begin{equation*}
\tilde{\phi}(\alpha )=\xi ^{\ast }\tilde{\phi}_{qn}\left( e\otimes \alpha
\right) \xi \leq \xi ^{\ast }\phi _{qn}\left( e\otimes \alpha \right) \xi
=\phi (\alpha )\text{.}
\end{equation*}%
This shows that $\tilde{\phi}|_{E}=\phi $.

(5)$\Rightarrow $(1): This is the same as (4)$\Rightarrow $(1).

(1)$\Rightarrow $(6): Consider a unital complete order embedding $\eta
:S\rightarrow B\left( H\right) $. By \cite[Theorem 2.3.1]{wittstock_ein_1981}
there exists a selfadjoint linear map $\psi :V\rightarrow M_{q}\left(
B\left( H\right) \right) $ such that $\psi |_{W}=\eta _{q}\circ \phi _{0}$
and $\psi \preccurlyeq \eta _{q}\circ \theta $. Since $S$ satisfies the weak
expectation property, there exists a unital completely positive map $\rho
:S\rightarrow S^{\ast \ast }$ such that $\eta \circ \rho :S\rightarrow
S^{\ast \ast }$ is the inclusion map. Therefore the map $\phi :=\rho
_{q}\circ \eta $ satisfies the desired conclusions.

(6)$\Rightarrow $(4): Obvious.
\end{proof}

Suppose now that $S$ is a function system, in which case $S\cong A(K)$ where 
$K$ is the state space of $S$. In this case, $S$ satisfies the weak
expectation property if and only if it is nuclear, which in turn is
equivalent to the assertion that $K$ is a simplex. Therefore in view of
Proposition \ref{Proposition:commutative}, Theorem \ref{Theorem:nc-Riesz}
can be seen as the noncommutative analog of the fact that a compact convex
set $K$ is a simplex if and only if $A(K)$ satisfies the Riesz separation
property.

\section{The weak expectation property and model theory\label{Section:model}}

\subsection{Positively existentially closed operator systems}

We now recall some notions from the logic for metric structures, in the
specific setting of operator systems. Let us say that a degree $1$ matrix
*-polynomial is an expression $p\left( x_{1},\ldots ,x_{n}\right) $ of the
form%
\begin{equation*}
\alpha _{1}^{\ast }x_{1}\beta _{1}+\cdots +\alpha _{n}^{\ast }x_{n}\beta
_{n}+\gamma _{1}^{\ast }1\delta _{1}+\cdots +\gamma _{n}^{\ast }1\delta _{n}
\end{equation*}%
where $\alpha _{i},\beta _{i},\gamma _{i},\delta _{i}$ are scalar matrices
of size $1\times d$ for some $d\in \mathbb{N}$. Let $S$ be an operator
system, and $a_{1},\ldots ,a_{n}$ be elements of $S$. Identifying $1$ in $%
p\left( x_{1},\ldots ,x_{n}\right) $ with the unit of $S$, one can regard $%
p\left( a_{1},\ldots ,a_{n}\right) $ as an element of $M_{d}(S)$ in a
natural way.

An \emph{atomic formula} is an expression of the form $\left\Vert p\left(
x_{1},\ldots ,x_{n}\right) \right\Vert $ for some degree $1$ matrix
*-polynomial. A positive \emph{positive quantifier-free formula} $\varphi
\left( x_{1},\ldots ,x_{n}\right) $ is an expression of the form%
\begin{equation*}
f\left( \varphi _{1}\left( x_{1},\ldots ,x_{n}\right) ,\ldots ,\varphi
_{k}\left( x_{1},\ldots ,x_{n}\right) \right) \text{,}
\end{equation*}%
where $f:\mathbb{R}^{n}\rightarrow \mathbb{R}$ is a continuous nondecreasing
functions and $\varphi _{1},\ldots ,\varphi _{k}$ are atomic formulas. Given
a positive quantifier-free formula $\varphi \left( x_{1},\ldots
,x_{n}\right) $, an operator system $S$, and elements $a_{1},\ldots ,a_{n}$
of $S$, one can define in a natural way the \emph{interpretation} $\varphi
\left( a_{1},\ldots ,a_{n}\right) $. We denote by $\mathrm{Ball}(S)$ the
unit ball of the operator system $S$.

\begin{definition}
\label{Definition:pec}An operator system $S$ is \emph{positively
existentially closed }in the class of operator systems if for any operator
system inclusion $S\subset T$, positive quantifier-free formula $\varphi
\left( x_{1},\ldots ,x_{n},y_{1},\ldots ,y_{m}\right) $, and $a_{1},\ldots
,a_{n}\in S$, one has that%
\begin{equation*}
\inf \left\{ \varphi \left( a_{1},\ldots ,a_{n},b_{1},\ldots ,b_{m}\right)
:b_{1},\ldots ,b_{m}\in \mathrm{Ball}(T)\right\} =\inf \left\{ \varphi
\left( a_{1},\ldots ,a_{n},b_{1},\ldots ,b_{m}\right) :b_{1},\ldots
,b_{m}\in \mathrm{\mathrm{\mathrm{Ball}}}(S)\right\} \text{.}
\end{equation*}
\end{definition}

Equivalently, one can say that $S$ is positively existentially closed if for
any operator system inclusion $S\subset T$, degree $1$ matrix *-polynomials $%
p_{1}\left( x_{1},\ldots ,x_{n},y_{1},\ldots ,y_{m}\right) ,\ldots
,p_{k}\left( x_{1},\ldots ,x_{n},y_{1},\ldots ,y_{m}\right) $, elements $%
a_{1},\ldots ,a_{n}\in \mathrm{Ball}(S)$ and $b_{1},\ldots ,b_{m}\in \mathrm{%
Ball}(T)$, and $\varepsilon >0$, there exist $c_{1},\ldots ,c_{m}\in \mathrm{%
Ball}(S)$ such that, for every $i\in \left\{ 1,2,\ldots ,k\right\} $,%
\begin{equation*}
\left\Vert p_{i}\left( a_{1},\ldots ,a_{n},c_{1},\ldots ,c_{m}\right)
\right\Vert \leq \left\Vert p_{i}\left( a_{1},\ldots ,a_{n},b_{1},\ldots
,b_{m}\right) \right\Vert +\varepsilon \text{.}
\end{equation*}

We now observe that, for an operator system, being positively existentially
closed is equivalent to having the weak expectation property. In the case of
unital C*-algebras, this is proved in \cite[Proposition 5.1]%
{goldbring_omitting_2015}.

\begin{proposition}
\label{Proposition:pec}Suppose that $S$ is an operator system. The following
statements are equivalent:

\begin{enumerate}
\item $S$ has the weak expectation property;

\item $S$ is positively existentially closed in the class of operator
systems.
\end{enumerate}
\end{proposition}

Proposition \ref{Proposition:pec} generalizes \cite[Proposition 4.18]%
{goldbring_omitting_2015}, where it is proved that for an exact operator
system being positively existentially closed is equivalent to being nuclear.

In the following, we consider the construction of a $q$-minimal operator
system from \cite{xhabli_super_2012}. Suppose that $S$ is an operator
system. Then the $q$-minimal operator system $\mathrm{OMIN}_{q}(S)$ is the
image of $S$ under the direct sum of all the unital completely positive maps 
$\phi :S\rightarrow M_{q}(\mathbb{C})$. The canonical surjective isomorphism 
$\eta :S\rightarrow \mathrm{OMIN}_{q}(S)$ is a unital completely positive
map such that $\eta _{q}:M_{q}(S)\rightarrow M_{q}\left( \mathrm{OMIN}%
_{q}(S)\right) $ is an order isomorphism. Observe that the assignment $%
S\mapsto \mathrm{OMIN}_{q}(S)$ is functorial, and any unital completely
positive map $\phi :S\rightarrow T$ induces a canonical unital completely
positive map $\mathrm{OMIN}_{q}\left( \phi \right) :\mathrm{OMIN}%
_{q}(S)\rightarrow \mathrm{OMIN}_{q}(T)$. By injectivity of $M_{q}(\mathbb{C}%
)$, if $\phi $ is a unital complete order embedding, then $\mathrm{OMIN}%
_{q}\left( \phi \right) $ is a unital complete order embedding. In
particular, an inclusion of operator systems $S\subset T$ induces a
canonical inclusion of operator systems $\mathrm{OMIN}_{q}(S)\subset \mathrm{%
OMIN}_{q}(T)$. Finally, $\mathrm{OMIN}_{q}\left( S^{\ast \ast }\right) $ is
unitally completely order isomorphic to $\mathrm{OMIN}_{q}(S)^{\ast \ast }$,
and if $\iota _{S}:S\rightarrow S^{\ast \ast }$ is the canonical inclusion
map, then $\mathrm{OMIN}_{q}\left( \iota _{S}\right) $ can be identified
with $\iota _{\mathrm{OMIN}_{q}(S)}$. We recall that an operator space $S$
is \emph{locally reflexive }if for any finite-dimensional operator space $F$
and complete contraction $\phi :F\rightarrow S^{\ast \ast }$, $\phi $ is the
point-$\sigma \left( S^{\ast \ast },S^{\ast }\right) $-limit of completely
contractions from $F$ to $S$; see \cite[Section 4]{effros_injectivity_2001}.

\begin{proof}[Proof of Proposition \protect\ref{Proposition:pec}]
The proof of \cite[Proposition 4.18]{goldbring_omitting_2015} shows that a
positively existentially closed operator system satisfies Condition\ (2) of
Proposition \ref{Proposition:characterization}. Indeed, the existence of a
unital completely positive map $\psi $ as in Condition (2) of Proposition %
\ref{Proposition:characterization} can be expressed in terms of the value of
a suitable quantifier-free formula. We now prove the converse.

Suppose that $S$ is an operator system satisfying the weak expectation
property. Suppose that $S\subset T$ is an operator system inclusion, $%
\varphi \left( x_{1},\ldots ,x_{n},y\right) $ is a positive quantifier-free
formula, $a_{1},\ldots ,a_{n}\in \mathrm{Ball}(S)$, $b_{1},\ldots ,b_{m}\in 
\mathrm{Ball}(T)$, and $\varepsilon >0$. We want to prove that there exist $%
c_{1},\ldots ,c_{m}\in \mathrm{Ball}(S)$ such that%
\begin{equation*}
\varphi \left( a_{1},\ldots ,a_{n},c_{1},\ldots ,c_{m}\right) \leq \varphi
\left( a_{1},\ldots ,a_{n},b_{1},\ldots ,b_{m}\right) +\varepsilon \text{.}
\end{equation*}%
Without loss of generality, we can assume that $T=B\left( H\right) $. Since $%
S$ satisfies the weak expectation property, there exists a unital completely
positive map $\psi :B\left( H\right) \rightarrow S^{\ast \ast }$ such that $%
\psi |_{S}$ is the inclusion map of $S$ into $S^{\ast \ast }$.

Set $F:=\mathrm{span}\left\{ a_{1},\ldots ,a_{n},b_{1},\ldots
,b_{m},1\right\} \subset B\left( H\right) $. One can choose $q\in \mathbb{N}$
large enough such that, denoting by $\eta :F\rightarrow \mathrm{OMIN}%
_{q}\left( F\right) $ the canonical linear isomorphism,%
\begin{equation*}
\varphi \left( \eta \left( a_{1}\right) ,\ldots ,\eta \left( a_{n}\right)
,\eta \left( b_{1}\right) ,\ldots ,\eta \left( b_{m}\right) \right) =\varphi
\left( a_{1},\ldots ,a_{n},b_{1},\ldots ,b_{m}\right) \text{.}
\end{equation*}

Consider now the unital completely positive map 
\begin{equation*}
\mathrm{OMIN}_{q}(\psi ):\mathrm{OMIN}_{q}\left( B\left( H\right) \right)
\rightarrow \mathrm{OMIN}_{q}\left( S^{\ast \ast }\right) \cong \mathrm{OMIN}%
_{q}(S)^{\ast \ast }\text{.}
\end{equation*}%
Observe that $\mathrm{OMIN}_{q}\left( B\left( H\right) \right) $ is exact
and, particularly, locally reflexive \cite[Corollary 14.6.5]%
{effros_operator_2000}. Therefore $\mathrm{OMIN}_{q}\left( \psi \right) $
can be approximated in the point-$\sigma \left( \mathrm{OMIN}_{q}(S)^{\ast
\ast },\mathrm{OMIN}_{q}(S)^{\ast }\right) $-topology by completely
contractive maps $\mathrm{OMIN}_{q}\left( F\right) \rightarrow \mathrm{OMIN}%
_{q}(S)$. Furthermore $\mathrm{O\mathrm{MIN}}_{q}\left( B\left( H\right)
\right) $ is unitally completely order isomorphic to a direct sum of copies
of $M_{q}(\mathbb{C})$. Fix $\delta >0$. A convexity argument as in the
proof of Proposition \ref{Proposition:characterization} shows that there
exists a unital completely positive map $\theta :\mathrm{O\mathrm{MIN}}%
_{q}\left( F\right) \rightarrow \mathrm{O\mathrm{MIN}}_{q}(S)$ such that $%
\left\Vert \theta |_{\mathrm{O\mathrm{MIN}}_{q}(E)}-\iota \right\Vert
<\delta $, where $\iota :\mathrm{O\mathrm{MIN}}_{q}(E)\rightarrow \mathrm{O%
\mathrm{MIN}}_{q}(S)$ is the inclusion map. By choosing $\delta $ small
enough one has that%
\begin{eqnarray*}
\varphi \left( a_{1},\ldots ,a_{n},\left( \theta \circ \eta \right) \left(
b_{1}\right) ,\ldots ,\left( \theta \circ \eta \right) \left( b_{m}\right)
\right)  &\leq &\varphi \left( \left( \theta \circ \eta \right) \left(
a_{1}\right) ,\ldots ,\left( \theta \circ \eta \right) \left( a_{n}\right)
,\left( \theta \circ \eta \right) \left( b_{1}\right) ,\ldots ,\left( \theta
\circ \eta \right) \left( b_{m}\right) \right) +\varepsilon  \\
&\leq &\varphi \left( \eta \left( a_{1}\right) ,\ldots ,\eta \left(
a_{n}\right) ,\eta \left( b_{1}\right) ,\ldots ,\eta \left( b_{m}\right)
\right) +\varepsilon  \\
&=&\varphi \left( a_{1},\ldots ,a_{n},b_{1},\ldots ,b_{m}\right)
+\varepsilon \text{.}
\end{eqnarray*}%
This concludes the proof that $S$ is positively existentially closed.
\end{proof}

\subsection{The weak expectation property and the complete tight Riesz
interpolation}

The complete tight Riesz interpolation property has been introduced for
unital C*-algebras by Kavruk in \cite{kavruk_weak_2012}. We consider here
its straightforward generalization to operator systems. Let $S\subset T$ be
an inclusion of operator systems. For elements $a,b$ of $M_{n}(T)$, write $%
a\ll b$ if $b-a\geq \delta 1$ for some $\delta >0$.

\begin{definition}[Kavruk]
The operator system $S$ has the \emph{relative complete tight Riesz
interpolation property }in $T$ if for any $n,k\in \mathbb{N}$, $x_{1},\ldots
,x_{k},y_{1},\ldots ,y_{k}\in M_{d}(S)_{\mathrm{sa}}$ and $b\in M_{n}(T)_{%
\mathrm{sa}}$ such that $x_{i}\ll b\ll y_{j}$ for every $i,j\in \left\{
1,2,\ldots ,k\right\} $, there exists $a\in M_{n}(S)_{\mathrm{sa}}$ such
that $x_{i}\leq a\leq y_{j}$ for every $i,j\in \left\{ 1,2,\ldots ,k\right\} 
$.
\end{definition}

It is proved in \cite[Theorem 7.4]{kavruk_weak_2012} that a unital
C*-algebra $A\subset B\left( H\right) $ satisfies the weak expectation
property if and only if it has the relative complete tight Riesz
interpolation property in $B\left( H\right) $. An alternative proof using
the characterization of the model-theoretic characterization of the weak
expectation property is presented in \cite[Theorem 5.5]%
{goldbring_omitting_2015}. In view of Proposition \ref{Proposition:pec},
such a proof applies equally well to operator systems, and it gives the
following.

\begin{proposition}
\label{Proposition:right}Suppose that $S\subset B\left( H\right) $ is an
operator system. The following assertions are equivalent.

\begin{enumerate}
\item $S$ satisfies the weak expectation property;

\item $S$ has the relative complete tight Riesz interpolation property in $%
B\left( H\right) $;

\item $S$ has the relative complete tight Riesz interpolation property in $T$
for some operator system inclusion $S\subset T$, where $T$ is an operator
system with the weak expectation property.
\end{enumerate}
\end{proposition}

\bibliographystyle{amsplain}
\bibliography{wep-biblio}

\providecommand{\MR}[1]{}\def\cprime{$'$}
\providecommand{\bysame}{\leavevmode\hbox to3em{\hrulefill}\thinspace}
\providecommand{\MR}{\relax\ifhmode\unskip\space\fi MR }
\providecommand{\MRhref}[2]{%
  \href{http://www.ams.org/mathscinet-getitem?mr=#1}{#2}
}
\providecommand{\href}[2]{#2}
\begin{thebibliography}{10}

\bibitem{alfsen_compact_1971}
Erik~M. Alfsen, \emph{Compact convex sets and boundary integrals},
  Springer-Verlag, New York-Heidelberg, 1971.

\bibitem{arveson_subalgebras_1969}
William Arveson, \emph{Subalgebras of {C}*-algebras}, Acta Mathematica
  \textbf{123} (1969), no.~1, 141--224.

\bibitem{arveson_subalgebras_1972}
\bysame, \emph{Subalgebras of {C}*-algebras {II}}, Acta Mathematica
  \textbf{128} (1972), no.~1, 271--308.

\bibitem{arveson_noncommutative_2008}
\bysame, \emph{The noncommutative {C}hoquet boundary}, Journal of the American
  Mathematical Society \textbf{21} (2008), no.~4, 1065--1084.

\bibitem{arveson_noncommutative_2010}
\bysame, \emph{The noncommutative {C}hoquet boundary {III}: operator systems in
  matrix algebras}, Mathematica Scandinavica \textbf{106} (2010), no.~2,
  196--210.

\bibitem{arveson_noncommutative_2011}
\bysame, \emph{The noncommutative {C}hoquet boundary {II}: hyperrigidity},
  Israel Journal of Mathematics \textbf{184} (2011), no.~1, 349--385.

\bibitem{barlak_sequentially_2016}
Sel{\c{c}}uk Barlak and G{\'{a}}bor Szab{\'{o}}, \emph{Sequentially split
  *-homomorphisms between {C}*-algebras}, International Journal of Mathematics
  \textbf{27} (2016), no.~13, 1650105, 48.

\bibitem{barlak_spatial_2017}
Sel{\c{c}}uk Barlak, G{\'{a}}bor Szab{\'{o}}, and Christian Voigt, \emph{The
  spatial {Rokhlin} property for actions of compact quantum groups}, Journal of
  Functional Analysis \textbf{272} (2017), no.~6, 2308--2360.

\bibitem{brown_c*-algebras_2008}
Nathanial~P. Brown and Narutaka Ozawa, \emph{{C}*-algebras and
  finite-dimensional approximations}, Graduate Studies in Mathematics, vol.~88,
  American Mathematical Society, Providence, {RI}, 2008.

\bibitem{choi_completely_1976}
Man-Duen Choi and Edward~G. Effros, \emph{The completely positive lifting
  problem for {C}*-algebras}, Annals of Mathematics \textbf{104} (1976), no.~3,
  585--609.

\bibitem{choi_injectivity_1977}
\bysame, \emph{Injectivity and operator spaces}, Journal of Functional Analysis
  \textbf{24} (1977), no.~2, 156--209.

\bibitem{choi_lifting_1977}
Man~Duen Choi and Edward~G. Effros, \emph{Lifting problems and the cohomology
  of {C}*-algebras}, Canadian Journal of Mathematics \textbf{29} (1977), no.~5,
  1092--1111.

\bibitem{choi_nuclear_1978}
Man-Duen Choi and Edward~G. Effros, \emph{Nuclear {C}*-algebras and the
  approximation property}, American Journal of Mathematics \textbf{100} (1978),
  no.~1, 61--79.

\bibitem{davidson_choquet_2015}
Kenneth~R. Davidson and Matthew Kennedy, \emph{The {Choquet} boundary of an
  operator system}, Duke Mathematical Journal \textbf{164} (2015), no.~15,
  2989--3004.

\bibitem{effros_aspects_1978}
Edward~G. Effros, \emph{Aspects of noncommutative order}, C*-algebras and
  applications to physics ({Proc}. {Second} {Japan}-{USA} {Sem}., {Los}
  {Angeles}, {Calif}., 1977), Lecture {Notes} in {Mathematics}, vol. 650,
  Springer, Berlin, 1978, pp.~1--40.

\bibitem{effros_injectivity_2001}
Edward~G. Effros, Narutaka Ozawa, and Zhong-Jin Ruan, \emph{On injectivity and
  nuclearity for operator spaces}, Duke Mathematical Journal \textbf{110}
  (2001), no.~3, 489--521.

\bibitem{effros_operator_2000}
Edward~G. Effros and Zhong-Jin Ruan, \emph{Operator spaces}, London
  Mathematical Society Monographs. New Series, vol.~23, Oxford University
  Press, 2000.

\bibitem{effros_operator_1997}
Edward~G. Effros and Corran Webster, \emph{Operator analogues of locally convex
  spaces}, Operator algebras and applications ({Samos}, 1996), {NATO} {Adv}.
  {Sci}. {Inst}. {Ser}. {C} {Math}. {Phys}. {Sci}., vol. 495, Kluwer Acad.
  Publ., Dordrecht, 1997, pp.~163--207.

\bibitem{farenick_characterisations_2013}
Douglas Farenick, Ali~S. Kavruk, Vern~I. Paulsen, and Ivan~G. Todorov,
  \emph{Characterisations of the weak expectation property}, arXiv:1307.1055
  (2013).

\bibitem{farenick_operator_2014}
\bysame, \emph{Operator {systems} from {discrete} {groups}}, Communications in
  Mathematical Physics \textbf{329} (2014), no.~1, 207--238.

\bibitem{farenick_c*-extreme_1997}
Douglas Farenick and Phillip Morenz, \emph{C*-extreme points in the generalized
  state spaces of a {C}*-algebra}, Transactions of the American Mathematical
  Society \textbf{349} (1997), no.~5, 1725--1748.

\bibitem{farenick_c*-extreme_1993}
Douglas Farenick and Phillip~B. Morenz, \emph{C*-extreme points of some compact
  {C}*-convex sets}, Proceedings of the American Mathematical Society
  \textbf{118} (1993), no.~3, 765--775.

\bibitem{farenick_operator_2012}
Douglas Farenick and Vern~I. Paulsen, \emph{Operator system quotients of matrix
  algebras and their tensor products}, Mathematica Scandinavica \textbf{111}
  (2012), no.~2, 210--243.

\bibitem{farenicK_extremal_2000}
Douglas~R. Farenick, \emph{Extremal matrix states on operator systems}, Journal
  of the London Mathematical Society \textbf{61} (2000), no.~3, 885--892.

\bibitem{farenick_pure_2004}
\bysame, \emph{Pure matrix states on operator systems}, Linear Algebra and its
  Applications \textbf{393} (2004), 149--173.

\bibitem{gardella_rokhlin_2017}
Eusebio Gardella, Mehrdad Kalantar, and Martino Lupini, \emph{Rokhlin dimension
  for compact quantum group actions}, arXiv:1703.10999 (2017).

\bibitem{gardella_equivariant_2016}
Eusebio Gardella and Martino Lupini, \emph{Equivariant logic and applications
  to {C}*-dynamics}, arXiv:1608.05532 (2016).

\bibitem{goldbring_model-theoretic_2015}
Isaac Goldbring and Martino Lupini, \emph{Model-theoretic aspects of the
  {G}urarij operator system}, Isreal Journal of Mathematics, in press.

\bibitem{goldbring_omitting_2015}
Isaac Goldbring and Thomas Sinclair, \emph{Omitting types in operator systems},
  Indiana University Mathematics Journal, to appear.

\bibitem{goldbring_kirchbergs_2015}
\bysame, \emph{On {Kirchberg}'s embedding problem}, Journal of Functional
  Analysis \textbf{269} (2015), no.~1, 155--198.

\bibitem{han_approximation_2011}
Kyung~Hoon Han and Vern~I. Paulsen, \emph{An approximation theorem for nuclear
  operator systems}, Journal of Functional Analysis \textbf{261} (2011), no.~4,
  999--1009.

\bibitem{hopenwasser_extreme_1981}
Alan Hopenwasser, Robert~L. Moore, and Vern~I. Paulsen, \emph{C*-extreme
  points}, Transactions of the American Mathematical Society \textbf{266}
  (1981), no.~1, 291--307.

\bibitem{junge_bilinear_1995}
Marius Junge and Gilles Pisier, \emph{Bilinear forms on exact operator spaces
  and {$B(H)\otimes B(H)$}}, Geometric and Functional Analysis \textbf{5}
  (1995), no.~2, 329--363.

\bibitem{kavruk_tensor_2011}
Ali Kavruk, Vern~I. Paulsen, Ivan~G. Todorov, and Mark Tomforde, \emph{Tensor
  products of operator systems}, Journal of Functional Analysis \textbf{261}
  (2011), no.~2, 267--299.

\bibitem{kavruk_weak_2012}
Ali~S. Kavruk, \emph{The {weak} {expectation} {property} and {Riesz}
  {interpolation}}, arXiv:1201.5414 (2012).

\bibitem{kavruk_quotients_2013}
Ali~S. Kavruk, Vern~I. Paulsen, Ivan~G. Todorov, and Mark Tomforde,
  \emph{Quotients, exactness, and nuclearity in the operator system category},
  Advances in Mathematics \textbf{235} (2013), 321--360.

\bibitem{kavruk_nuclearity_2014}
Ali~Samil Kavruk, \emph{Nuclearity related properties in operator systems},
  Journal of Operator Theory \textbf{71} (2014), no.~1, 95--156.

\bibitem{lazar_spaces_1968}
Aldo~J. Lazar, \emph{Spaces of affine continuous functions on simplexes},
  Transactions of the American Mathematical Society \textbf{134} (1968), no.~3,
  503--525.

\bibitem{loebl_remarks_1981}
Richard~I. Loebl and Vern~I. Paulsen, \emph{Some remarks on {C}*-convexity},
  Linear Algebra and its Applications \textbf{35} (1981), 63--78.

\bibitem{lupini_fraisse_2015}
Martino Lupini, \emph{Fra\"{i}ss\'{e} limits in functional analysis},
  arXiv:1510.05188 (2015).

\bibitem{paulsen_completely_1982}
Vern~I. Paulsen, \emph{Completely bounded maps on {C}*-algebras and invariant
  operator ranges}, Proceedings of the American Mathematical Society
  \textbf{86} (1982), no.~1, 91--96.

\bibitem{paulsen_completely_1984}
\bysame, \emph{Completely bounded homomorphisms of operator algebras},
  Proceedings of the American Mathematical Society \textbf{92} (1984), no.~2,
  225--228.

\bibitem{paulsen_completely_2002}
\bysame, \emph{Completely bounded maps and operator algebras}, Cambridge
  Studies in Advanced Mathematics, vol.~78, Cambridge University Press,
  Cambridge, 2002.

\bibitem{pisier_introduction_2003}
Gilles Pisier, \emph{Introduction to operator space theory}, London
  Mathematical Society Lecture Note Series, vol. 294, Cambridge University
  Press, Cambridge, 2003.

\bibitem{schmitt_characterization_1985}
Lothar~M. Schmitt, \emph{Characterization of {$L^2(\mathcal{M})$} for injective
  {W}*-algebras {$\mathcal{M}$}}, Mathematica Scandinavica \textbf{57} (1985),
  no.~2, 267--280.

\bibitem{schmitt_characterization_1982}
Lothar~M. Schmitt and Gerd Wittstock, \emph{Characterization of matrix-ordered
  standard forms of {W}*-algebras}, Mathematica Scandinavica \textbf{51}
  (1982), no.~2, 241--260.

\bibitem{webster_krein-milman_1999}
Corran Webster and Soren Winkler, \emph{The {Krein}-{Milman} theorem in
  operator convexity}, Transactions of the American Mathematical Society
  \textbf{351} (1999), no.~1, 307--322.

\bibitem{winkler_non-commutative_1999}
Soren Winkler, \emph{The non-commutative {Legendre}-{Fenchel} transform},
  Mathematica Scandinavica \textbf{85} (1999), no.~1, 30--48.

\bibitem{wittstock_ein_1981}
Gerd Wittstock, \emph{Ein operatorwertiger {H}ahn-{B}anach {S}atz}, Journal of
  Functional Analysis \textbf{40} (1981), no.~2, 127--150.

\bibitem{wittstock_extension_1984}
\bysame, \emph{Extension of completely bounded {C}*-module homomorphisms},
  Operator algebras and group representations, {Vol}. {II} ({Neptun}, 1980),
  Monogr. {Stud}. {Math}., vol.~18, Pitman, Boston, MA, 1984, pp.~238--250.

\bibitem{wittstock_matrix_1984}
\bysame, \emph{On matrix order and convexity}, Functional analysis: surveys and
  recent results, {III} ({Paderborn}, 1983), North-{Holland} {Math}. {Stud}.,
  vol.~90, North-Holland, Amsterdam, 1984, pp.~175--188.

\bibitem{xhabli_super_2012}
Blerina Xhabli, \emph{The super operator system structures and their
  applications in quantum entanglement theory}, Journal of Functional Analysis
  \textbf{262} (2012), no.~4, 1466--1497.

\end{thebibliography}

\end{document}